\documentclass[11pt,a4paper]{article}
\usepackage[margin=2.5cm]{geometry}
\usepackage{fontspec}
\usepackage{amsmath,amssymb,amsthm}
\usepackage{hyperref}
\usepackage[capitalise]{cleveref}

\newtheorem{theorem}{Theorem}[section]
\newtheorem{lemma}[theorem]{Lemma}
\newtheorem{corollary}[theorem]{Corollary}

\newcommand{\cR}{\mathcal R}
\newcommand{\e}{\mathrm e}

\hypersetup{
  colorlinks=true,
  linkcolor=blue,
  citecolor=blue,
  urlcolor=blue,
  pdftitle={Uniform Minimum-Degree Guarantees in Hypergraph Edge-Coloring Games},
  pdfauthor={}
}

\title{A One-Third Bound for the Maker--Breaker Degree Game}
\author{
Kaizhe Chen\thanks{Department of Mathematics, Princeton University, Princeton, USA. Email: kc8556@princeton.edu}
~~~~
Xiaonan Chen\thanks{Department of Mathematics, University of California Irvine, Irvine, USA. Email: xiaonac1@uci.edu}
~~~~
Haotian Yang\thanks{Department of Mathematics, California Institute of Technology, Pasadena, USA. Email: hyang3@caltech.edu}
}
\date{}

\begin{document}
\maketitle

\begin{abstract}
Let \(H\) be a finite hypergraph with rank $r$ and minimum degree \(d\). In the Maker--Breaker degree game, Maker and Breaker alternately claim previously unclaimed hyperedges of \(H\), with Maker moving first, and Breaker seeks to maximize the minimum degree $\delta(H_{\mathrm B})$ of his spanning subhypergraph $H_{\mathrm B}$. We prove that, for every fixed \(r\ge2\) and all sufficiently large \(d\), Breaker has a deterministic strategy satisfying
\[
\delta(H_{\mathrm B})\ge \frac{d}{r+1}-\sqrt{d\log d}.
\]
For graphs, this improves the classical universal lower bound \(d/4\) to \(d/3-\sqrt{d\log d}\).
The proof introduces a virtual balancing process, encodes local imbalance by a multiplicative risk, and keeps the resulting risk vector in the Shearer region throughout the game.
\end{abstract}

\section{Introduction}

Maker--Breaker games form one of the central classes of positional games. 
A Maker--Breaker game is specified by a finite board and a prescribed family of winning sets. Maker and Breaker alternately claim previously unclaimed elements of the board, with Maker moving first. 
Maker seeks to claim every element of some winning set, whereas Breaker seeks to prevent this. For general background, see \cite{BeckBook,HKSZBook,KrivelevichSurvey}.

In this paper, we study the following Maker--Breaker degree game.
Let \(H=(V,E)\) be a finite hypergraph. 
For \(v\in V\), the \emph{degree} \(d_H(v)\) of $v$ is the number of hyperedges containing \(v\). Denote by \(\delta(H)\) the minimum degree of $H$.
Maker and Breaker alternately claim previously unclaimed hyperedges of \(H\), with Maker moving first, until every hyperedge has been claimed.
Let \(H_{\mathrm M}\) and \(H_{\mathrm B}\) be the spanning subhypergraphs formed by the hyperedges claimed by Maker and Breaker, respectively.
Breaker seeks to maximize \(\delta(H_{\mathrm B})\), while Maker seeks to minimize it.
Let \(\delta_{\mathrm B}(H)\) denote the largest integer \(q\) such that Breaker has a strategy guaranteeing \(\delta(H_{\mathrm B})\ge q\) at the end of the game.

The degree game is much better understood on dense boards. For the complete graph \(K_n\), Beck~\cite{BeckDeterministic} proved that 
\(\delta_{\mathrm B}(K_n)\le n/2-\Omega(\sqrt n)\), while Alon, Krivelevich, Spencer, and Szab\'o~\cite{DiscrepancyGames} proved that \(\delta_{\mathrm B}(K_n)\ge n/2-O(\sqrt{n\log n})\). In fact, applying their discrepancy-game theorem \cite[Theorem~1]{DiscrepancyGames} gives \(\delta_{\mathrm B}(H)\ge d/2-O(\sqrt{d\log n})\) for every graph $H$ on $n$ vertices with minimum degree $d$. 
In particular, \(\delta_{\mathrm B}(H)\ge d/2-o(d)\) whenever \(\log n=o(d)\). The main difficulty therefore lies in obtaining a bound independent of the number of vertices.

For arbitrary graph $H$ with minimum degree $d$, the best previously known lower bound was the classical estimate \(\delta_{\mathrm B}(H)\ge \lfloor d/4\rfloor\).
A standard orientation-and-pairing strategy (see, e.g., \cite{HKSZdegree}) proves this bound: orient the graph so that every vertex \(v\) has outdegree at least \(\lfloor d_H(v)/2\rfloor\), and then partition the outgoing edges at each vertex into pairs. Whenever Maker claims one edge of a pair, Breaker claims the other. This guarantees $\delta(H_{\mathrm B})\ge \lfloor d/4\rfloor$ at the end of the game. 
Despite its simplicity, this $1/4$ bound has not been improved for decades. In his ICM survey, Krivelevich \cite[Section~8]{KrivelevichSurvey} remarked that ``even improving it to \( (1/4+\varepsilon)d\) would be quite nice''. 
Beck \cite[Open Problem~16.1]{BeckBook} also singled out this problem in his monograph, listing it as the first of his seven ``most humiliating'' open problems in positional game theory.

For $d$-regular graphs, Beck \cite{BeckSurplus} records the following ``probabilistic intuition'' of Szab\'o: after each Maker move, Breaker chooses one endpoint of the edge claimed by Maker uniformly at random and tries to claim an edge incident with that endpoint. 
At the end of the game, if Maker has claimed \(x\) edges incident with a vertex \(v\), then one expects Breaker to have chosen \(v\) roughly \(x/2\) times. Since 
$x+\frac{x}{2}\approx d$, we have \(x/2\approx d/3\). This suggests $d/3$ as an important threshold for the degree game.

Recently, Gy\H{o}rffy~\cite{GyorffyDegree} realized the \(d/3\) bound for several special families of regular graphs. He showed that Breaker can guarantee degree at least \(\lfloor d/3\rfloor\) on the \(d\)-dimensional hypercube, with analogous bounds for grids and tori. 

The \emph{rank} of a hypergraph $H$ is the maximum cardinality of an edge in $H$.
In this note, we prove the \(d/3\) bound for every graph of minimum degree $d$ and, more generally, obtain the corresponding \(d/(r+1)\) bound for hypergraphs of rank $r$, in both cases up to a lower-order error term.
\begin{theorem}\label{main-theorem}
For every fixed integer \(r\ge2\), there exists a constant \(d_0(r)\) such that, for every finite hypergraph $H$ with rank $r$ and minimum degree \(d\ge d_0(r)\),
\[
\delta_{\mathrm B}(H)\ge \frac{d}{r+1}-\sqrt{d\log d}.
\]
\end{theorem}

In particular, this gives the following consequence for graphs.

\begin{corollary}
There exists a constant \(d_0\) such that, for every finite graph $H$ with minimum degree \(d\ge d_0\),
\[
\delta_{\mathrm B}(H)\ge \frac d3-\sqrt{d\log d}.
\]
\end{corollary}

The key idea of the proof is to run a virtual balancing process that, after each Maker move, selects one of the affected vertices at which Breaker attempts to respond. We associate with each vertex a multiplicative risk measuring the imbalance between the number of times it is selected and the number of times it is affected but not selected. An imbalance large enough to threaten the desired degree bound forces the risk to be at least \(1\). We maintain throughout the game that the resulting risk vector lies in the Shearer region of the \(2\)-section graph. Since membership in this region forces every coordinate to be less than \(1\), no such imbalance can occur. This dynamic use of the Shearer region replaces a one-shot local-lemma argument and introduces no dependence on the number of vertices.

Related degree and balancing games have been widely studied. 
We mention only a few directions here.
Degree and pseudorandomness games on complete graph boards were studied in~\cite{DiscrepancyGames,JumbleG}, while related minimum-degree games on general and sparse graph boards appear in~\cite{BaloghPluhar,HKSZdegree}. Gy\H{o}rffy recently considered the corresponding Chooser--Picker degree game on regular graphs~\cite{GyorffyChooser}; see also~\cite{Knox} for limitations of general comparisons between Maker--Breaker and Chooser--Picker games. Bennett, Frieze, and Pegden~\cite{BennettFriezePegden} studied a related biased minimum-degree game on complete uniform hypergraphs. Generalized pairing strategies~\cite{GyorffyPluhar}, discrepancy minimization~\cite{BansalDiscrepancy}, and equitable orientations of uniform hypergraphs~\cite{CaroWestYuster,CohenLochet} provide further related balancing frameworks.

\section{The Shearer region}

Throughout this section, let \(\mathcal G\) be a finite graph with vertex set \(V\). For any \(z=(z_v)_{v\in V}\in[0,\infty)^V\) and \(S\subseteq V\), define
\[
Q_S(z):=\sum_{\substack{I\subseteq S\\ I\text{ independent in }\mathcal G}}
(-1)^{|I|}\prod_{v\in I}z_v.
\]
In particular, \(Q_\varnothing(z)=1\). The \emph{Shearer region} of \(\mathcal G\) is defined as
$$\cR(\mathcal G):=\{z\in[0,\infty)^V:Q_S(z)>0\text{ for every }S\subseteq V\}.$$
We next present several basic properties of the Shearer region that will be used in the proof; see Scott and Sokal~\cite{ScottSokal} for background.

For any \(v\in V\), let \(\Gamma(v)\) denote the set of neighbors of \(v\) in \(\mathcal G\). For any $S\subseteq V\setminus\{v\}$, separating the independent sets of \(S\cup\{v\}\) according to whether they contain \(v\) gives
\begin{equation}\label{deletion-recurrence}
Q_{S\cup\{v\}}(z)=Q_S(z)-z_v Q_{S\setminus\Gamma(v)}(z),\qquad {\rm for\  any}\  z\in[0,\infty)^V.
\end{equation}
For \(v\in V\), \(S\subseteq V\setminus\{v\}\), and \(z\in[0,\infty)^V\) with \(Q_S(z)>0\), put \(R_v(z,S):=Q_{S\setminus\Gamma(v)}(z)/Q_S(z)\).

\begin{lemma}\label{ratio-monotonicity}
If \(z\in\cR(\mathcal G)\), \(v\in V\), and \(S\subseteq T\subseteq V\setminus\{v\}\), then \(R_v(z,S)\le R_v(z,T)\).
\end{lemma}

\begin{proof}
It suffices to consider \(T=S\cup\{u\}\), where \(u\notin S\cup\{v\}\). We induct on \(|S|\). For the base case \(S=\varnothing\), if \( u\notin\Gamma(v)\), then \(R_v(z,\{u\})=1\); if \(u\in\Gamma(v)\), then \(R_v(z,\{u\})=1/(1-z_u)\). In both cases, $R_v(z,\{u\})\ge 1=R_v(z,\varnothing)$. This proves the base case. Now, suppose that \(S\ne\varnothing\) and that the assertion holds for every set of cardinality less than \(|S|\).

If \(u\in\Gamma(v)\), then \((S\cup\{u\})\setminus\Gamma(v)=S\setminus\Gamma(v)\), while \eqref{deletion-recurrence} gives \(Q_{S\cup\{u\}}(z)\le Q_S(z)\). Hence, \(R_v(z,S)\le R_v(z,S\cup\{u\})\), since the numerator is unchanged, while the positive denominator does not increase.

Now, suppose that \(u\notin\Gamma(v)\). Applying \eqref{deletion-recurrence} to the numerator and denominator of $R_v(z,S\cup\{u\})$ gives
\[
\frac{R_v(z,S\cup\{u\})}{R_v(z,S)}=\frac{Q_{S\setminus\Gamma(v)}(z)-z_u Q_{S\setminus(\Gamma(v)\cup \Gamma(u))}(z)}{Q_S(z)-z_u Q_{S\setminus\Gamma(u)}(z)} \cdot \frac{Q_S(z)}{Q_{S\setminus\Gamma(v)}(z)}
=\frac{1-z_u R_u(z,S\setminus\Gamma(v))}{1-z_u R_u(z,S)}.
\]
The denominator in the last fraction equals \(Q_{S\cup\{u\}}(z)/Q_S(z)\) and is positive. By adding the vertices of \(S\cap\Gamma(v)\) one by one and applying the induction hypothesis, we obtain \( R_u(z,S\setminus\Gamma(v)) \le R_u(z,S)\). The last fraction is therefore at least \(1\), completing the induction.
\end{proof}

\begin{lemma}\label{clique-section}
Let \(C\) be a clique in \(\mathcal G\) and put \(A:=V\setminus C\). Fix \(z_A=(z_v)_{v\in A}\in[0,\infty)^A\), and extend it to a vector \(z_0=(z_v)_{v\in V}\) by setting \(z_v=0\) for every \(v\in V\setminus A\). Assume that \(Q_B(z_0)>0\) for every \(B\subseteq A\). For every \(x=(x_v)_{v\in C}\in[0,\infty)^C\), let \((z_A,x)\in[0,\infty)^V\) denote the vector whose restrictions to \(A\) and \(C\) are \(z_A\) and \(x\), respectively. Then \((z_A,x)\) lies in \(\cR(\mathcal G)\) if and only if
\begin{equation}\label{clique-section-inequality}
\sum_{v\in C}x_vR_v(z_0,A)<1.
\end{equation}
\end{lemma}

\begin{proof}
For every \(S\subseteq V\), we have \(Q_S(z_0)=Q_{S\cap A}(z_0)>0\), so \( z_0\in\cR(\mathcal G)\). If \((z_A,x)\in\cR(\mathcal G)\), then \(Q_V(z_A,x)>0\). Since \(C\) is a clique, an independent set in \(V\) contains at most one vertex of \(C\). Hence
\[
Q_V(z_A,x)
=Q_A(z_0)-\sum_{v\in C}x_v Q_{A\setminus\Gamma(v)}(z_0)
=Q_A(z_0)\left(1-\sum_{v\in C}x_v R_v(z_0,A)\right),
\]
which proves \eqref{clique-section-inequality}.

On the other hand, suppose that \eqref{clique-section-inequality} holds. 
To prove \((z_A,x)\in\cR(\mathcal G)\), it suffices to show that $Q_{B\cup D}(z_A,x)>0$ for any \(B\subseteq A\) and \(D\subseteq C\). Similarly, since $D$ is a clique, we have
\[
Q_{B\cup D}(z_A,x)=Q_B(z_0)-\sum_{v\in D}x_v Q_{B\setminus\Gamma(v)}(z_0)
=Q_B(z_0)\left(1-\sum_{v\in D}x_v R_v(z_0,B)\right).
\]
By Lemma \ref{ratio-monotonicity}, $R_v(z_0,B)\le R_v(z_0,A)$ for any $v\in C$. So, \eqref{clique-section-inequality} yields 
\[\sum_{v\in D}x_v R_v(z_0,B)\le\sum_{v\in C}x_v R_v(z_0,A)<1.\]
It follows that \(Q_{B\cup D}(z_A,x)>0\). This completes the proof.
\end{proof}

Write \(\mathbf{1}_V\) for the constant vector in \(\mathbb R^V\) whose coordinates are all equal to \(1\).

\begin{lemma}\label{symmetric-shearer}
Suppose that \(\Delta(\mathcal G)\le D\), where \(D\ge1\). Then, for any \(p\ge 0\) with \(\e p(D+1)\le 1\), we have \(p \mathbf{1}_V \in \cR(\mathcal G)\).
\end{lemma}

\begin{proof}
Put \(x:=1/(D+1)\). We prove by induction on \(|S|\) that
$Q_S(p\mathbf{1}_V)>0$ for every $S\subseteq V$ and $Q_S(p\mathbf{1}_V)\ge(1-x)Q_{S\setminus\{v\}}(p\mathbf{1}_V)$ for every $v\in S$.
For \(S=\varnothing\), the first assertion follows from \(Q_\varnothing(p\mathbf{1}_V)=1\), while the second is vacuous.
Now, suppose that \(S\ne\varnothing\) and that the two assertions hold for every set of cardinality less than \(|S|\).

Fix \(v\in S\), put \(T:=S\setminus\{v\}\), and suppose that \(v\) has \(h\) neighbors in \(T\). Repeatedly applying the induction hypothesis while deleting these $h$ neighbors gives 
\[Q_T(p\mathbf{1}_V)\ge(1-x)^h Q_{T\setminus\Gamma(v)}(p\mathbf{1}_V)>0.\] 
It follows from \eqref{deletion-recurrence} that
\[
\frac{Q_S(p\mathbf{1}_V)}{Q_T(p\mathbf{1}_V)}=1-\frac{pQ_{T\setminus\Gamma(v)}(p\mathbf{1}_V)}{Q_T(p\mathbf{1}_V)}
\ge 1- \frac{p}{(1-x)^{h}}  
\ge 1-\frac{p}{(1-x)^{D}}= 1-p\left(1+\frac{1}{D}\right)^D\ge 1-\e p,
\]
where we have used the fact that $h\le D$.
By the assumption that \(\e p\le x\), the last expression is at least \(1-x>0\). This completes the induction and hence the lemma.
\end{proof}

\begin{lemma}\label{clique-average}
Let \(z=(z_v)_{v\in V}\in\cR(\mathcal G)\), and let \(C\) be a clique in \( \mathcal G\). Let \(s\) be a positive integer, and for each \(j\in\{1,\ldots,s\}\), let \(z^{(j)}=(z_v^{(j)})_{v\in V}\in[0,\infty)^V\).
Suppose that \(z_v^{(j)}=z_v\) for every \(v\in V\setminus C\) and \(j\in\{1,\ldots,s\}\), and that 
\begin{equation}\label{24}
\sum_{j=1}^s z_v^{(j)}\le z_v s 
\end{equation}
for every \(v\in C\). Then, there exists \(j\in\{1,\ldots,s\}\) such that \( z^{(j)}\in\cR(\mathcal G)\).
\end{lemma}

\begin{proof}
Put \(A:=V\setminus C\). Let \(z_0\) be obtained from \(z\) by replacing the coordinates in \(C\) with \(0\). Note that \(Q_S(z_0)=Q_{S}(z)>0\) for any \( S\subseteq A\).  By Lemma \ref{clique-section}, we have \(\sum_{v\in C}z_v R_v(z_0,A)<1\). It follows from \eqref{24} that
\[
\sum_{j=1}^s\sum_{v\in C}z_v^{(j)}R_v(z_0,A)=\sum_{v\in C} \sum_{j=1}^s z_v^{(j)}R_v(z_0,A)
\le \sum_{v\in C}z_v s R_v(z_0,A)<s.
\]
Thus, there exists \(j\in\{1,\ldots,s\}\) such that \(\sum_{v\in C}z_v^{(j)}R_v(z_0,A)<1\). Applying Lemma \ref{clique-section} again gives \( z^{(j)}\in\cR(\mathcal G)\).
\end{proof}

\section{Proof of the main theorem}

Fix \(r\ge 2\) and let \(H=(V,E)\) be a finite hypergraph with rank $r$ and sufficiently large minimum degree \(d\). For every \(v\in V\), choose a set \( F_v\subseteq E\) consisting of exactly \(d\) hyperedges incident with \(v\). 
At the end of the game, write \(d_{\mathrm M}^F(v):=|F_v\cap E(H_{\mathrm M})|\) and \(d_{\mathrm B}^F(v):=|F_v\cap E(H_{\mathrm B})|\). Then, \(d_{\mathrm M}^F(v)+d_{\mathrm B}^F(v)=d\) for any $v\in V$.
Set \[\eta:=\sqrt{\frac{3\log d}{d}}\qquad {\rm  and }\qquad m:=\left\lceil\left(\frac{r-1}{r+1}+\eta\right)d\right\rceil.\]

A \emph{full round} consists of a Maker move followed by a Breaker move. We now define a virtual process that is performed between these two moves (Maker's possible final unmatched move causes no virtual update). Initially, every vertex $v\in V$ is \emph{active} and has counters \(x_v=y_v=0\).
Suppose that Maker claims \(e\in E\) in a full round, and let
\[
C(e):=\{v\in e:v\text{ is active and }e\in F_v\}.
\]
If \(C(e)=\varnothing\), no counter is changed. Otherwise, before claiming his hyperedge, Breaker \emph{selects} a vertex \(h\in C(e)\), increases \(y_h\) by \(1\), and increases \(x_v\) by \(1\) for every \(v\in C(e)\setminus\{h\}\). The rule for selecting \(h\) will be specified later.
After this update, for any active vertex \(v\in V\), we declare $v$ \emph{safe} if \(x_v+2y_v\ge d\); if \(x_v+2y_v< d\) but \(x_v=m\), we declare it \emph{dangerous}; otherwise, it remains active. Once a vertex becomes safe or dangerous, its counters are frozen.

\begin{lemma}\label{virtual-to-real}
If Breaker has a deterministic rule for making the virtual choices under which no vertex ever becomes dangerous, then he has a deterministic strategy such that, for every \(v\in V\),
$$d_{\mathrm B}^F(v)\ge \frac{d-m}{2}-1.$$
\end{lemma}

\begin{proof}
Fix an ordering of \(E\). Suppose that Maker claims \(e\in E\) in a full round. If \(C(e)\ne\varnothing\) and the virtual choice is \(v\), Breaker claims the first unclaimed hyperedge in \(F_v\), provided one exists. In all other cases, he claims the first unclaimed hyperedge of \(H\). Such a hyperedge exists because the move occurs in a full round.

For any fixed \(v\in V\), the final value of \(y_v\) records the number of full rounds in which \(v\) is selected. Within these \(y_v\) full rounds, there is at most one in which Breaker claims no hyperedge in \( F_v\). Indeed, after such a full round every hyperedge in \(F_v\) has already been claimed, so no later Maker move can cause \(v\) to be selected again. 
Therefore, \(d_{\mathrm B}^F(v)\ge y_v-1\).

If \(v\) becomes safe, then \(x_v+2y_v\ge d\). Moreover, \(x_v\le m\), since \(v\) did not become dangerous. 
Thus, \(2y_v\ge d-x_v\ge d-m\).
The desired estimate then follows by the inequality \(d_{\mathrm B}^F(v)\ge y_v-1\).

It remains to consider the case where \(v\) is active until the end. 
Every time Maker claims a hyperedge in \(F_v\), $C(e)\ne \emptyset$, so either \(x_v\) or \(y_v\) increases by \(1\). Thus, $d_{\mathrm M}^F(v)\le x_v+y_v+1$, where 1 comes from Maker's possible final unmatched move.
At the end of the game, since $v$ becomes neither safe nor dangerous, \( x_v<m\) and \(x_v+2y_v<d\), and hence \(x_v+y_v<(m+d)/2\). So, 
$d_{\mathrm M}^F(v)\le 1+(m+d)/2$, and the desired estimate follows from the fact \(d_{\mathrm M}^F(v)+d_{\mathrm B}^F(v)=d\).
\end{proof}

\noindent Define 
\[ 
\beta:=\frac{r-1}{r},\qquad
\ell:=\left\lfloor\frac{d+m-1}{2}\right\rfloor,\qquad \gamma:=\frac m\ell,\qquad a:=\frac{\gamma}{\beta}, \qquad {\rm and} \qquad
b:=\frac{1-\gamma}{1-\beta}.
\]
For sufficiently large $d$, we have \(m\le d-3\) and hence \(\ell>m\). Moreover,
\begin{equation}\label{rbeta}
\gamma-\beta\ge\frac{2m}{d+m}-\frac{r-1}{r}
=\frac{(r+1)m-(r-1)d}{r(d+m)}\ge \frac{(r+1)\eta d}{2rd}=\frac{(r+1)\eta}{2r},
\end{equation}
where the last inequality follows from the definition of $m$ and the fact that $m<d$. In particular, we obtain \(0<\beta<\gamma<1\), and hence \(a>1>b>0\).

For every active vertex \(v\), define its \emph{risk} by \( z_v:=a^{x_v-m}b^{y_v+m-\ell}\). 
Initially, the risk of every vertex equals \(p:=a^{-m}b^{m-\ell}\).
When a vertex $v$ ceases to be active, its risk $z_v$ is frozen along with $x_v$ and $y_v$. Since \(m=\gamma\ell\), we have
\[
-\log p = \ell\left(\gamma\log\frac\gamma\beta+(1-\gamma)\log\frac{1-\gamma}{1-\beta}\right) = \int^\gamma_\beta \frac{\ell(\gamma -t)}{t(1-t)}dt
\ge \int^\gamma_\beta 4\ell(\gamma -t)dt= 2\ell (\gamma-\beta)^2. 
\]
By \eqref{rbeta} and the fact that $\ell >m\ge (r-1)d/(r+1)$, we derive
\begin{equation}\label{p}
-\log p\ge \frac{(r^2-1)\eta^2d}{2r^2}
\ge \frac{3\eta^2d}{8}.
\end{equation}

\begin{lemma}\label{dangerous}
    For any $v\in V$, if \(v\) becomes dangerous, then \(z_v\ge1\).
\end{lemma}
\begin{proof}
    If \(v\) becomes dangerous, then \(x_v+2y_v<d\) and \(x_v=m\), so \(y_v\le\lfloor(d-m-1)/2\rfloor=\ell-m\). Since \(b<1\), the definition of \(z_v\) gives \(z_v=b^{y_v+m-\ell}\ge 1\).
\end{proof}

\begin{lemma}\label{risk-properties}
Suppose that Maker claims an edge \(e\in E\) in a full round when the current risk vector is \(z=(z_v)_{v\in V}\), and that \(C(e)\ne\varnothing\). For each \(h\in C(e)\), let \(z^{(h)}=(z^{(h)}_v)_{v\in V}\) denote the risk vector resulting from the virtual update in which \(h\) is selected. Then, for every \(v\in C(e)\),
\begin{equation}\label{average}
    \sum_{h\in C(e)}z_v^{(h)}\le z_v|C(e)|.
\end{equation}
\end{lemma}

\begin{proof}
Put \(s:=|C(e)|\le r\) and fix \(v\in C(e)\). Among the \(s\) possible choices of $h$, \(y_v\) is increased once and \(x_v\) is increased \(s-1\) times.
Note that increasing \(y_v\) multiplies \(z_v\) by \(b\), whereas increasing \(x_v\) multiplies it by \(a\). Hence
\[
\sum_{h\in C(e)}z_v^{(h)}
=\left(b+(s-1)a \right)z_v \le \left( \frac{sb}{r}+\frac{(r-1)sa}{r}\right)z_v =\left( (1-\beta)b +\beta a \right)z_v s=z_v s,
\]
where the inequality follows from the fact that \(a>b\).
\end{proof}

\begin{proof}[Proof of \cref{main-theorem}]
Define a graph \(\mathcal G\) on \(V\) by joining distinct vertices \(u\) and \(v\) whenever \(F_u\cap F_v\ne\varnothing\). Then, for any \(e\in E\), the set \(\{v\in e:e\in F_v\}\) is a clique in \(\mathcal G\). For any \(v\in V\), each neighbor of \(v\) in \(\mathcal G\) lies with \(v\) in some hyperedge of \(F_v\), and hence \(d_{\mathcal G}(v)\le (r-1)|F_v|= (r-1)d\). Therefore, \(\Delta({\mathcal G})\le(r-1)d\).

Initially, the risk vector \(z=p\mathbf{1}_V\). By \eqref{p}, we have 
$$\e p((r-1)d+1)\le \e((r-1)d+1)\exp\left(-\frac{3\eta^2d}{8}\right)= \e((r-1)d+1) d^{-9/8} \le 1,$$ 
for sufficiently large $d$.
Since \(\Delta({\mathcal G})\le(r-1)d\), applying Lemma \ref{symmetric-shearer} with \(D=(r-1)d\) shows that the initial risk vector $p\mathbf{1}_V$ belongs to \(\cR({\mathcal G})\).

Suppose that the current risk vector \(z\) lies in \(\cR(\mathcal G)\) and that Maker claims a hyperedge \(e\in E\) in a full round.
If \(C(e)=\varnothing\), then \(z\) does not change. Otherwise, \(C(e)\subseteq \{v\in e:e\in F_v\}\) is a clique in \({\mathcal G}\). 
By Lemma \ref{risk-properties}, \eqref{average} holds for any $v\in C(e)$. 
Together with the fact that $z_v$ does not change for any $v\in V\setminus C(e)$, Lemma \ref{clique-average} shows that Breaker can select some $h\in C(e)$ such that $z^{(h)}\in \cR({\mathcal G})$. Thus, the risk vector $z$ remains in \(\cR(\mathcal G)\) after every virtual update

During the game, since \(z\in\cR({\mathcal G})\), we have \(1-z_v=Q_{\{v\}}(z)>0\) for every \(v\in V\). It follows from Lemma \ref{dangerous} that no vertex ever becomes dangerous. By Lemma \ref{virtual-to-real}, Breaker has a deterministic strategy such that, for every \( v\in V\),
\[
d_{\mathrm B}^F(v) \ge\frac{d-m}{2}-1
\ge\frac{d}{r+1}-\frac{\eta d}{2}-\frac32
=\frac{d}{r+1}-\frac{\sqrt{3d\log d}}{2}-\frac32> \frac{d}{r+1}-\sqrt{d\log d},
\]
where we have used the definition of $m$ and $\eta$.
Since \(d_{H_{\mathrm B}}(v)\ge d_{\mathrm B}^F(v)\) for any $v\in V$ at the end of the game, the same bound holds for $\delta(H_{\mathrm B})$.
This completes the proof.
\end{proof}

\section{Concluding remarks}
In this paper, we proved that Breaker can guarantee $\delta_B(H)\ge d/(r+1)-o(d)$ on every finite hypergraph $H$ with rank $r$ and minimum degree \(d\). 
The lower-order loss \(\sqrt{d\log d}\) in our proof comes from the Shearer initialization: a slack of \(\eta d\) gives initial risk \(\exp(-\Theta(\eta^2d))\), while the 2-section graph $\mathcal G$ has degree \(O_r(d)\).  Thus feasibility requires \(\eta^2d=\Theta(\log d)\), yielding the error term. 

The main remaining question is naturally to determine the correct asymptotic value of
the degree game.  Even for \(d\)-regular graphs, there is a substantial
gap between our \(d/3-o(d)\) lower bound and the natural upper bound \(d/2\).


\section*{Acknowledgments} 
The authors thank Professor Asaf Ferber for suggesting this problem at the 2025 Desert Discrete Mathematics Workshop, where this project began. The workshop was supported by NSF CAREER DMS-2146406.
K. Chen's contribution to this work was partially completed while he was visiting the Institute for Mathematical Sciences, National University of Singapore, during the program ``Innovations and Challenges in Extremal Combinatorics''.

\section*{Statement on AI}
The authors used AI as an exploratory tool to assist in developing the arguments. The authors have carefully checked all mathematical proofs and take full responsibility for the content of the manuscript.

\bibliographystyle{abbrv}
\bibliography{references}

@article{HKSZdegree,
  author  = {Hefetz, Dan and Krivelevich, Michael and Stojakovi{\'c}, Milo\v{s} and Szab{\'o}, Tibor},
  title   = {A sharp threshold for the {Hamilton} cycle {Maker--Breaker} game},
  journal = {Random Structures \& Algorithms},
  volume  = {34},
  number  = {1},
  pages   = {112--122},
  year    = {2009}
}

@article{BeckDeterministic,
  author  = {Beck, J{\'o}zsef},
  title   = {Deterministic Graph Games and a Probabilistic Intuition},
  journal = {Combinatorics, Probability and Computing},
  volume  = {3},
  number  = {1},
  pages   = {13--26},
  year    = {1994}
}

@book{HKSZBook,
  author    = {Hefetz, Dan and Krivelevich, Michael and
               Stojakovi{\'c}, Milo{\v s} and Szab{\'o}, Tibor},
  title     = {Positional Games},
  series    = {Oberwolfach Seminars},
  volume    = {44},
  publisher = {Birkh{\"a}user},
  address   = {Basel},
  year      = {2014}
}

@incollection{KrivelevichSurvey,
  author    = {Krivelevich, Michael},
  title     = {Positional Games},
  booktitle = {Proceedings of the International Congress of Mathematicians---Seoul 2014},
  editor    = {Jang, Sun Young and Kim, Young Rock and Lee, Dae-Woong and Yie, Ikkwon},
  volume    = {4},
  pages     = {355--379},
  publisher = {Kyung Moon Sa},
  address   = {Seoul},
  year      = {2014}
}

@article{DiscrepancyGames,
  title={Discrepancy games},
  author={Alon, Noga and Krivelevich, Michael and Spencer, Joel and Szab{\'o}, Tibor},
  journal={The Electronic Journal of Combinatorics},
  volume={12},
  number={1},
  pages={R51},
  year={2005},
  doi={10.37236/1948}
}

@article{CaroWestYuster,
  title={Equitable hypergraph orientations},
  author={Caro, Yair and West, Douglas B. and Yuster, Raphael},
  journal={The Electronic Journal of Combinatorics},
  volume={18},
  number={1},
  pages={P121},
  year={2011},
  doi={10.37236/608}
}

@article{GyorffyPluhar,
  title={Generalized pairing strategies---a bridge from pairing strategies to colorings},
  author={Gy{\H{o}}rffy, Lajos and Pluh{\'a}r, Andr{\'a}s},
  journal={Acta Universitatis Sapientiae, Mathematica},
  volume={8},
  number={2},
  pages={233--248},
  year={2016},
  doi={10.1515/ausm-2016-0015}
}

@misc{Knox,
  title={Two Constructions Relating to Conjectures of {Beck} on Positional Games},
  author={Knox, Fiachra},
  year={2012},
  note={arXiv:1212.3345}
}

@article{ScottSokal,
  author={Scott, Alexander D. and Sokal, Alan D.},
  title={The Repulsive Lattice Gas, the Independent-Set Polynomial, and the {Lov\'asz} Local Lemma},
  journal={Journal of Statistical Physics},
  volume={118},
  number={5--6},
  pages={1151--1261},
  year={2005}
}

@article{CohenLochet,
  title={Equitable orientations of sparse uniform hypergraphs},
  author={Cohen, Nathann and Lochet, William},
  journal={The Electronic Journal of Combinatorics},
  volume={23},
  number={4},
  pages={P4.31},
  year={2016},
  doi={10.37236/6152}
}

@misc{GyorffyDegree,
  title={Degree Game for Special Regular Graphs},
  author={Gy{\H{o}}rffy, Lajos},
  year={2026},
  note={arXiv:2608.11007}
}

@misc{GyorffyChooser,
  title={Chooser-Picker Degree Games for Regular Graphs},
  author={Gy{\H{o}}rffy, Lajos},
  year={2026},
  note={arXiv:2608.11035}
}

@inproceedings{BansalDiscrepancy,
  author={Bansal, Nikhil},
  title={Constructive Algorithms for Discrepancy Minimization},
  booktitle={Proceedings of the 51st Annual {IEEE} Symposium on Foundations of Computer Science},
  pages={3--10},
  publisher={IEEE Computer Society},
  year={2010}
}

@misc{BennettFriezePegden,
  author        = {Patrick Bennett and Alan Frieze and Wesley Pegden},
  title         = {Some Maker-Breaker Games on Hypergraphs},
  year={2025},
  note={arXiv:2509.02788}
}

@article{JumbleG,
  author  = {Alan Frieze and Michael Krivelevich and Oleg Pikhurko
             and Tibor Szab{\'o}},
  title   = {The Game of JumbleG},
  journal = {Combinatorics, Probability and Computing},
  volume  = {14},
  pages   = {783--793},
  year    = {2005},
  doi     = {10.1017/S0963548305006851}
}

@book{BeckBook,
  author    = {J{\'o}zsef Beck},
  title     = {Combinatorial Games: Tic-Tac-Toe Theory},
  series    = {Encyclopedia of Mathematics and its Applications},
  volume    = {114},
  publisher = {Cambridge University Press},
  address   = {Cambridge},
  year      = {2008},
  doi       = {10.1017/CBO9780511735202},
  isbn      = {978-0-521-46100-9}
}

@article{BaloghPluhar,
  author  = {J\'ozsef Balogh and Andr\'as Pluh\'ar},
  title   = {The Positive Minimum Degree Game on Sparse Graphs},
  journal = {The Electronic Journal of Combinatorics},
  volume  = {19},
  number  = {1},
  pages   = {P22},
  year    = {2012},
  doi     = {10.37236/1174}
}

@incollection{BeckSurplus,
  author    = {J{\'o}zsef Beck},
  title     = {Surplus of Graphs and the Lov{\'a}sz Local Lemma},
  booktitle = {Building Bridges: Between Mathematics and Computer Science},
  editor    = {Gr{\"o}tschel, Martin and Katona, Gyula O. H. and S{\'a}gi, G{\'a}bor},
  series    = {Bolyai Society Mathematical Studies},
  volume    = {19},
  pages     = {47--102},
  publisher = {Springer},
  year      = {2008}
}
\end{document}